\documentclass[11pt]{amsart}
\usepackage{ae}
\usepackage[all]{xy}
\usepackage{graphicx,amsfonts,amssymb,amsmath}
\usepackage{natbib}

\usepackage{pifont}

 \usepackage{hyperref}
\hypersetup{colorlinks,linkcolor=blue,citecolor=red,}

\usepackage[OT2,T1]{fontenc}
\DeclareSymbolFont{cyrletters}{OT2}{wncyr}{m}{n}
\DeclareMathSymbol{\sha}{\mathalpha}{cyrletters}{"58}
\input cyracc.def

 \newtheorem{thm}{Theorem}[section]
 \newtheorem{ques}[thm]{Question}

 \newtheorem{lem}[thm]{Lemma}
 \newtheorem{conj}[thm]{Conjecture}
 \newtheorem{prop}[thm]{Proposition}
 \theoremstyle{definition}
 
 \theoremstyle{remark}
 
 \theoremstyle{remark}
 \newtheorem{rem}[thm]{Remark}
 \numberwithin{equation}{subsection}

 \newcommand{\To}{\longrightarrow}

 \newcommand{\E}{\mathcal{E}}

 \renewcommand{\ker}{\textup{Ker}}
 \newcommand{\Pic}{\textup{Pic}}
 \newcommand{\Br}{\textup{Br}}
 \renewcommand{\H}{\textup{H}}

 \newcommand{\BM}{Brauer\textendash Manin\ }
 \newcommand{\BMo}{Brauer\textendash Manin obstruction\ }
  \newcommand{\eBM}{\'etale-Brauer\textendash Manin\ }

 \renewcommand{\a}{\mathfrak{a}}
 
 \renewcommand{\P}{\mathbb{P}}

 \newcommand{\p}{\mathfrak{p}}

 \newcommand{\A}{\textbf{A}}
 \newcommand{\Q}{\mathbb{Q}}

 \newcommand{\Z}{\mathbb{Z}}
 
 \newcommand{\m}{\mathfrak{m}}
 \renewcommand{\O}{\mathcal{O}}

\begin{document}

\title[Weak approximation]
 {Non-invariance of  weak approximation properties under extension of the ground field}

\author{Yongqi Liang}



\keywords{weak approximation, \BM obstruction, extension of the ground field}
\thanks{\textit{MSC 2010} : 11G35 14G05  14G25 14J20}


\maketitle

\begin{abstract} 
For rational points on algebraic varieties defined over a number field $K$, we study the behavior of the property of weak approximation 
with Brauer\textendash Manin obstruction under extension of the ground field. 
We construct $K$-varieties accompanied with a quadratic extension $L|K$ such that the property holds over $K$ (conditional on a conjecture)  while fails over $L$. The result is unconditional when $K=\Q$ or $K$ is one of several quadratic number fields.
Over $\Q$, we give an explicit example.
\end{abstract}


\section{Introduction}
Historically, Chinese remainder theorem has been extended to weak approximation on algebraic groups and on more general algebraic varieties.
Let $K$ be a number field. We consider smooth proper algebraic varieties $X$ defined over $K$. \emph{Weak approximation} means that the image of the diagonal embedding of rational points $ X(K)\subset\prod_{v} X(K_{v})$ is dense with respect to the product topology of $v$-adic topologies.
Analytic methods (e.g. Hardy\textendash Littlewood circle method) and cohomological obstructions have been developed to the study of weak approximation.

In 1970s, Manin defined the so-called Brauer\textendash Manin pairing between $\prod_{v}X(K_{v})$ and the cohomological  Brauer group $\Br(X)=\H^{2}_{\mbox{\scriptsize\'et}}(X,\mathbb{G}_{m})$. The \emph{Brauer\textendash Manin set} $[\prod_{v}X(K_{v})]^{\Br}$ is the subset of families of local points that are orthogonal to the Brauer group. The inclusions $X(K)\subset[\prod_{v}X(K_{v})]^{\Br}\subset\prod_{v}X(K_{v})$ obstruct weak approximation if the second inclusion is strict. We say that $X$ verifies \emph{weak approximation with Brauer\textendash Manin obstruction} if $X(K)$ is dense in the Brauer\textendash Manin set. Further obstructions to weak approximation have also been  studied, for example the \eBM or equivalently the descent obstruction. Similarly, strong approximation properties with respect to the adelic topology on not necessarily proper varieties are also considered in the literature.

The following natural question is of general interest for a property P on a variety. 
In this paper we will discuss arithmetic cases in which the property P is one of the following:
\newline \indent - weak approximation;
\newline \indent - weak approximation with \BM obstruction;
\newline \indent - weak approximation with \eBM obstruction.
\begin{ques}\label{mainquestion}
If $X$ verifies a property P, does $X\otimes_{K}L$ verify the same property under any finite extension $L$ of the ground field $K$?
\end{ques}
To the knowledge of the author, for weak approximation properties it has not yet been seriously discussed in the literature.

For \emph{strong} approximation with \BM obstruction, in a collaboration of Y. Cao, F. Xu and the author \cite[the paragraph after Corollary 8.2]{CLX}, we found   examples of punctured Abelian varieties to give a negative answer to the question  over $K=\Q$, and over an arbitrary number field conditionally on the finiteness of the Tate\textendash Shafarevich group of certain Abelian varieties. 
However, if the concerned Tate\textendash Shafarevich groups are finite, such examples can \textbf{not} be served as  a negative answer to the question for \emph{weak} approximation with \BM obstruction according to the Cassels\textendash Tate exact sequence for Abelian varieties and the purity theorem for the Brauer group. Nevertheless, we may expect the answer to be negative.

Leading to a negative answer to the question for weak approximation with \BM obstruction (off $\infty$),
it will not be easy to find examples within the following families of smooth proper varieties:  
\newline \indent - curves; 
\newline \indent - rational surfaces, or even rationally connected varieties.
\newline Indeed, respectively by V. Scharaschkin \cite{Scharaschkin}, A. Skorobogatov \cite[\S 6.2]{Skbook}, M. Stoll \cite[Conjecture 9.1]{Stoll07}, by J.-L. Colliot-Th\'el\`ene and J.-J. Sansuc \cite[page 174]{CT03}, the \BMo is suggested or conjectured to be the only obstruction to weak approximation over \emph{arbitrary} number fields for varieties in each family above. A. Skorobogatov \cite{Sk-beyond}, B. Poonen \cite{Poonen}, and many others \cite{HarpazSk,CTPSk,Smeets14} found varieties which do not verify weak approximation with \BM (or further) obstruction(s). But they did not consider the (non-)invariance of such properties under extension of the ground field.

Concerning weak approximation properties, to avoid the case of less interest, varieties appeared in this paper admit at least one rational point over the ground field.

\textbf{Main results.}
The  purpose of this paper is to construct examples to  answer negatively to the question. 
First of all, making use of the \BM obstruction, we construct Ch\^atelet surfaces  over  an arbitrary number field $K$ such that they verify weak approximation over $K$ but fail weak approximation over a quadratic extension $L$ of $K$, cf. Theorem \ref{exWA}.
Secondly, using such Ch\^atelet surfaces with the idea of B. Poonen \cite{Poonen}, we construct varieties that verify weak approximation with \BMo (or further obstructions) over $K$ but fail the same property over $L$.

\begin{thm}[Theorem \ref{exBMWA}]\label{mainthm}
Let $K$ be a number field. There exists a 3-fold $X$ over $K$ and a quadratic extension $L$ of $K$ such that 
\newline \indent - $X(\A_{K})^{\Br}\neq\emptyset$, and  if Conjecture \ref{Stollconj} is assumed to hold over $K$ then $X_{K}$ verifies weak approximation with \BMo off $\infty_{K}$;
\newline \indent - $X_{L}$ does not verify weak approximation with \BM obstruction off $\infty_{L}$. 
\end{thm}

For a general $K$, the first  conclusion is conditional. When $K=\Q$ or some quadratic number fields  ($\Q(\sqrt{-1})$, $\Q(\sqrt{3})$ etcetera) our result is unconditional.
In particular, we give an explicit example over $\Q$ in \S \ref{Qexample}. It is given by gluing affine pieces, one of which is defined by the equations
$$\left\{ \begin{array}{rcl}
y^{2}-17z^{2} &=&137\left[(1+y'^{2})^{2}(x^{4}+10x^{2}-155) - x'^{2}y'^{2}(x^{4}-155) \right]\\
y'^{2} &=& x'^{3}-4x'
\end{array} \right. $$
in $(x,y,z,x',y')\in\mathbb{A}^5$. It satisfies unconditionally the conclusions of Theorem \ref{mainthm} with $L = \Q(\sqrt{5})$.

\section{Preliminaries}\label{preliminaries}
In this paper, $K$ is always a number field. 
For completeness, we recall some results in class field theory.  All statements in this section are well known.

We denote by $\Omega_{K}$ the set of places of $K$ and by $\infty_{K}\subset\Omega_{K}$ the subset of archimedean places. For $v\in\Omega_{K}$, $K_{v}$ is the completion of $K$. The ring of ad\`eles (resp. ad\`eles without archimedean components) is denoted by $\A_{K}$ (resp. $\A_{K}^{\infty}$). The natural projection is denoted by $pr^{\infty}:\A_{K}\to\A_{K}^{\infty}$.

If an element $a$ in the ring of integer $\O_{K}$ generates a prime ideal, we will denote the prime ideal and the associated valuation  respectively by $\p_{a}$ and $v_{a}$.

Applying the Chebotarev density theorem and global class field theory to a ray class field, we obtain the following generalised Dirichlet's theorem on arithmetic progressions.

\begin{prop}\label{dirichlet}
Let $\a_{i}\subset\O_{K}(i=1,\ldots,s)$  be ideals that are pairwise prime to each other. Let $x_{i}\in\O_{K}$ be an element that is invertible in $\O_{K}/\a_{i}$.
Then there exists a principal prime ideal $\p=(p)\subset\O_{K}$ such that 
\begin{itemize}
\item[-] $p\equiv x_{i}\mod \a_{i}$ for all $i$;
\item[-] $p$ is positive and sufficiently large with respect to all real places of $K$.
\end{itemize}
In fact, the Dirichlet density of such principal prime ideals is positive.
\end{prop}

\begin{proof}
Let $\m_{0}=\prod_{i}\a_{i}=\bigcap_{i}\a_{i}$. According to the Chinese Remainder Theorem, there exists an element $x_{0}\in\O_{K}$ invertible in $\O_{K}/\m_{0}$ such that $x_{0}\equiv x_{i}\mod \a_{i}$ for all $i$. Furthermore, one can require  that $x_{0}$ is positive and sufficiently large with respect to all real places of $K$. 

Let $\m=\m_{0}\m_{\infty}$ be a modulus of $K$, where $\m_{\infty}$ is the product of all real places of $K$. Let $K_{\m_{}}$ be the ray class field of $K$ of modulus $\m$. Let $I^{\m}$ be the group of fractional ideals of $K$ that are prime to $\m_{0}$. Let $\psi:I^{\m}\to \textup{Gal}(K_{\m}|K)$ be the Artin reciprocity map. 
Define $K_{\m,1}$ to be the set of all elements $a\in K^{*}$ such that $v_{\p}(a-1)\geq v_{\p}(\m_{0})$ for all places $\p|\m_{0}$ and such that $a$ is positive with respect to all real places of $k$. The ray class group $C_{\m}$ is the cokernel of $i:K_{\m,1}\to I^{\m}$ which associates to an element $a\in K_{\m,1}$ the principal ideal $(a)$. Global class field theory says that $\psi$ is surjective with kernel $\ker(\psi)=i(K_{\m,1})$, cf. \cite[Chapter VI]{NeukirchANT}.  According to the Chebotarev density theorem \cite[Theorem VII.13.4]{NeukirchANT}, there exists a prime ideal $\p\nmid\m_{0}$ mapping to $\psi((x_{0}))\in\textup{Gal}(K_{\m}|K)$. Then $\p=(x_{0})\cdot(x)\in I^{\m}$ for a certain $x\in K_{\m,1}$. One checks that  $p=x_{0}\cdot x$ satisfies the conditions. Such prime ideals form a subset of $\Omega_{K}$ of positive Dirichlet density.
\end{proof}

\begin{lem}\label{hilbertsymbol}
Let $v$ be an odd place of $K$. For $s,t\in\O_{K}$ with $v(s)=0$, then the Hilbert symbol $(s,t)_{v}=-1$ if and only if $v(t)$ is odd and $s$ is not a square modulo $v$.
\end{lem}
\begin{proof}
It follows from \cite[Proposition V.3.4]{NeukirchANT}.
\end{proof}

\begin{lem}[Quadratic reciprocity law]\label{reciprocity}
Let $s,t\in\O_{K}$ be elements generating odd prime ideals. Assume that either  $s\equiv1\mod8\O_{K}$ or $t\equiv1\mod8\O_{K}$ and assume that for each real place either $s$ or $t$ is positive. Then $s$ is a square modulo $\p_{t}$ if and only if $t$ is a square modulo $\p_{s}$.
\end{lem}
\begin{proof}
By Lemma \ref{hilbertsymbol}, $s$ is a square modulo $\p_{t}$ if and only if $(s,t)_{v_{t}}=1$. And $t$ is a square modulo $\p_{s}$ if and only if $(s,t)_{v_{s}}=1$. It suffices to show that $(s,t)_{v_{s}}=(s,t)_{v_{t}}$. For any real place $v$  we have $(s,t)_{v}=1$ by hypothesis. For any $v|2$, either $s$ or $t$ is a square in $K_{v}$ by Hensel's lemma, hence $(s,t)_{v}=1$. For any finite place $v\nmid2st$ we have $(s,t)_{v}=1$ by Lemma \ref{hilbertsymbol}. The statement follows from the product formula $\displaystyle\prod_{v\in\Omega_{K}}(s,t)_{v}=1$.
\end{proof}

\section{Examples for non-invariance of weak approximation}\label{exWAsectiontitle}
V. A. Iskovskikh  \cite{Iskov71} showed that the Ch\^{a}telet surface over $\Q$ given by $$y^{2}+z^{2}=(x^{2}-2)(3-x^{2})$$
violates the Hasse principle. Summarising many others' work, A. Skorobogatov
\cite[page 145]{Skbook} studied variations of such a counterexample to Hasse principle over $\Q$.  B. Poonen  \cite[Proposition 5.1]{Poonen09} further generalised the argument to constructions over any number field $K$. A fortiori these examples do not verify weak approximation over the ground field $K$. In this section, we prove the following statement.

\begin{thm}\label{exWA}
Let $K$ be a number field. There exists a Ch\^{a}telet surface $V$ over $K$ and a quadratic extension $L$ of $K$ such that 
\newline \indent - $V_{K}$ verifies weak approximation and $V(\A_{K})\neq\emptyset$;
\newline \indent - $V_{L}$ does not verify weak approximation. 
\end{thm}

\begin{rem}
For the second conclusion above, we will prove a slightly stronger statement: $V_{L}$ does not verify weak approximation off $\infty_{L}$, cf. Proposition \ref{prop2}, which will be needed in the next section.
\end{rem}

Let $V_{0}\subset\mathbb{A}^{3}$ be the affine surface over $K$ defined by the equation $$y^{2}-az^{2}=b(x^{4}+2cx^{2}+d)\leqno (\star)$$
with $a,b,c,d\in\O_{K}$.
We define a Ch\^{a}telet $K$-surface $V$ as a smooth compactification of $V_{0}$ over $K$.
As the property of weak approximation with \BMo is  birationally invariant between smooth proper geometrically rational varieties \cite[Proposition 6.1]{CTPSk}, the forthcoming discussion will not depend on the choice of the compactification.
We are going to choose the parameters $a,b,c$ and $d$ to ensure  that  $D=c^{2}-d$ is not a square in $K$ and that $V$ verifies the theorem with $L=K(\sqrt{D})$ a quadratic extension of $K$, cf. Propositions \ref{prop1} and \ref{prop2}.

\noindent\textbf{Choice of  the parameters}

Applying Proposition \ref{dirichlet}, we can choose sequentially $a,b,c,e\in\O_{K}$ generating different prime ideals   such that
\begin{enumerate}
\item  $a\equiv1\mod8\O_{K}$ and $a>0$ with respect to all real places; 
\item  $b\equiv1\mod2a\O_{K}$; 
\item $c\equiv1\mod2\O_{K}$ and $c$ is not a  square modulo $\p_a$; 
\item $e$ is not a square modulo $\p_a$.
\end{enumerate}
Set $d=ce$. Then $D=c^{2}-d=c(c-e)$ is not a square in $K$ since $v_{c}(D)=1$.
 We find that 
\begin{itemize}
\item $a$ is a square modulo $\p_{b}$ by Lemma \ref{reciprocity};
\item $a$ is not a square modulo $\p_c$ by Lemma \ref{reciprocity};
\item $bd=bce$ is a square modulo $\p_{a}$ by (2)(3)(4).
\end{itemize}

\begin{prop}\label{prop1}
$V_{K}$ verifies weak approximation and $V_{K}$ admits rational points locally everywhere, i.e. $\overline{V(K)}=V(\A_{K})\neq\emptyset$.
\end{prop}

\begin{proof}
By Eisenstein's criterion, the polynomial $b(x^{4}+2cx^{2}+d)$ is irreducible over the local field $K_{\p_{c}}$, hence irreducible over $K$. Then $\Br(V_{K})/\Br(K)=0$ and $V_{K}$ satisfies weak approximation \cite[Theorem 8.11]{chateletsurfaces}.

It remains to  show that $V_{0}(K_{v})\neq\emptyset$ for all $v$.

As $bd$ is a square modulo $\p_{a}$, hence it is a square in $K_{v_{a}}$ by Hensel's lemma. By setting $z=x=0$ we get a  $K_{v_{a}}$-point on $V_{0}$. 

For any $v|2$ or $v=v_b$, $a$ is a square in $K_{v}$ by Hensel's lemma, $V_{0}$ is then $K_{v}$-rational. 
$V_{0}$ is also $K_{v}$-rational for all real $v$ since $a$ is positive. 

For any $v\nmid2ab$, if $a$ is a square in $K_{v}$ then $V_{0}$ is $K_{v}$-rational. Suppose that $a$ is not a square in $K_{v}$. Consider $x\in K_{v}$ such that $v(x)<0$, we find that  $v(b(x^{4}+2cx^{2}+d))=4v(x)$ is always even. As the extension $K_{v}(\sqrt{a})|K_{v}$ is unramified, $b(x^{4}+2cx^{2}+d)$ is a norm by  \cite[Corollary V.1.2]{NeukirchANT}, hence $V_{0}$ admits a $K_{v}$-rational point.

In summary $V(\A_{K})\neq\emptyset$.
\end{proof}

\begin{prop}\label{prop2}
$V_{L}$ does not verify weak approximation off $\infty_{L}$, i.e. $V(L)$ is not dense in $V(\A_{L}^{\infty})$.
\end{prop}

\begin{proof}
Over $L=K(\sqrt{D})$, the polynomial $x^{4}+2cx^{2}+d$ factorizes as $(x^2+c+\sqrt{D})(x^2+c-\sqrt{D})$. The affine surface $V_{0,L}$  is defined by $$y^{2}-az^{2}=b(x^2+c+\sqrt{D})(x^2+c-\sqrt{D}).$$

Consider $A\in\Br(L(V))$ the class of  the quaternion algebra $(a,x^{2}+c+\sqrt{D})$. As classes of quaternion algebras are of order $2$ in the Brauer group, $A$ equals to the class of $(a,1+\frac{c+\sqrt{D}}{x^{2}})$.
It equals  also to the class of $(a,b(x^{2}+c-\sqrt{D}))$ since $b(x^{2}+c-\sqrt{D})(x^{2}+c+\sqrt{D})=y^{2}-az^{2}$ lies in the image of the norm map $N_{L(V)(\sqrt{a})|L(V)}$. These three representations show that the residue of $A$ vanishes  at every point of codimension $1$ of $V_{L}$. Whence $A\in\Br(V_{L})$. Indeed, $A$ generates $\Br(V_{L})/\Br(L)\simeq \Z/2\Z$ \cite[Proposition 7.1.2]{Skbook}, but this will not help our argument.

Since $c|D$ but $c^2\nmid D$, $X^2-D\in K_{\p_c}[X]$ is an Eisenstein polynomial. 
The prime $\p_c$ is ramified  in the quadratic extension $L|K$. Denote by $\mathfrak{P}$ the prime ideal above $\p_c$. 
To show that weak approximation does not hold on $V_L$, according to the \BM pairing, it suffices to show that the evaluation of $A$ on $V_{L}(L_{\mathfrak{P}})$ takes both  $0$ and $1/2\in\Q/\Z$ as values. This is to say that there exist different points  $(x,y,z)\in V_{L}(L_{\mathfrak{P}})$ such that the Hilbert symbol $(a,x^{2}+c+\sqrt{D})_{\mathfrak{P}}$ can take both $1$ and $-1$ as values. As $\O_L/\mathfrak{P}=\O_K/{\p_c}$, $a$ is not a square modulo $\mathfrak{P}$. According to Lemma \ref{hilbertsymbol}, $(a,x^{2}+c+\sqrt{D})_{\mathfrak{P}}=1$ (resp. $-1$) if and only if $v_\mathfrak{P}(x^{2}+c+\sqrt{D})$ is even (resp. odd). 
Remember that $L_\mathfrak{P}(\sqrt{a})|L_\mathfrak{P}$ is unramified, $V_L$ admits $L_\mathfrak{P}$-points of coordinate $x$ if $v_\mathfrak{P}(b(x^4+2cx^2+d))$ is even \cite[Corollary V.1.2]{NeukirchANT}. 
Firstly, if $v_\mathfrak{P}(x)<0$, then both $v_\mathfrak{P}(b(x^4+2cx^2+d))=v_\mathfrak{P}(x^{4})$ and $v_\mathfrak{P}(x^{2}+c+\sqrt{D})=v_\mathfrak{P}(x^{2})$ are even. Therefore there exist $L_\mathfrak{P}$-points $(x,y,z)$ such that $(a,x^{2}+c+\sqrt{D})_{\mathfrak{P}}=1$.
Secondly,  as $v_\mathfrak{P}(c)=2$ implies that $v_\mathfrak{P}(\sqrt{D})=v_\mathfrak{P}(\sqrt{c(c-e)})=1$, we find that $v_\mathfrak{P}(c^2+c+\sqrt{D})=1$ is odd and $v_\mathfrak{P}(b(c^4+2c\cdot c^2+d))=2$ is even. Hence there exist   $L_\mathfrak{P}$-points $(x,y,z)=(c,y,z)$ such that $(a,x^{2}+c+\sqrt{D})_{\mathfrak{P}}=-1$. Therefore $\overline{V(L)}\subset V(\A_{L})^{\Br(V_{L})}\subsetneq V(\A_{L})$, weak approximation fails on $V_{L}$.

Moreover, the pull-back of the  class $A$ to $\Br(V_{L_{w}})$ is trivial for any archimedean place $w$ of $L$ since $a$ is a square in $L_{w}$. Thus archimedean places  make no contribution to the \BM pairing with $A$. Whence $pr^{\infty}(V(\A_{L})^{\Br(V_{L})})\subsetneq V(\A_{L}^{\infty})$, weak approximation off $\infty_{L}$ also fails on $V_{L}$.
\end{proof}

\section{Examples for non-invariance of weak approximation with Brauer\textendash Manin obstruction}\label{exBMWAsectiontitle}

The aim of this section is to adapt Poonen's construction \cite{Poonen} to find a $K$-variety $X$ such that it verifies weak approximation with \BMo off $\infty_{K}$ but $X_{L}$ fails the same property. In general, the  conclusion over $K$ is conditional on Conjecture \ref{Stollconj}. When $K=\Q$ we will give an explicit unconditional example. We will also show that over several quadratic number field $K$ the result is unconditional.

Over an arbitrary number field $F$, Poonen \cite{Poonen} constructed a fibration $X$ in Ch\^atelet surfaces over a curve possessing only finitely many rational points such that $X$ violates the Hasse principle with \'etale-\BM obstruction over the ground field. 
First of all, we recall his construction over $F$. Eventually $F$ will be the quadratic extension $L$ of $K$ in the previous section. We are able to make the construction be defined over the smaller field $K$. Moreover, we will show that the arithmetic of the fibration over $K$ behaves in contrast to that over $L$ if a key fiber is chosen to be one of the Ch\^{a}telet surfaces descripted in the previous section.

\subsection{Construction of Poonen}\label{poonenconstruction}\ \\
Start with a Ch\^{a}telet surface $Y_{\infty}$ defined over $F$ by the  equation $$y^{2}-az^{2}=P_{\infty}(x)$$with $a\in F^{*}$ and $P_{\infty}(x)\in F[x]$ a separable  polynomial of degree $4$. At the end, we will tell Poonen's choice of $Y_\infty$ and its consequences.  In the next paragraph, we will specify $F$ and a different suitable $Y_\infty$  for our purpose.

Take another separable polynomial $P_{0}(x)$ of degree $4$ and relatively prime to $P_{\infty}$. Let $\tilde{P}_{\infty}(w,x)$ and $\tilde{P}_{0}(w,x)$ be their homogenizations. 
Take a section  $$s'=u^{2}\tilde{P}_{\infty}(w,x)+v^{2}\tilde{P}_{0}(w,x)\in\Gamma(\P^{1}\times\P^{1},\O(1,2)^{\otimes2})$$
where $\P^{1}\times\P^{1}$ has homogeneous coordinates $(u:v,w:x)$. Let $Z'$ be the zero locus of $s'$ and let $R\subset\P^{1}$ be the (finite) branch locus of the first projection $Z'\subset\P^{1}\times\P^{1}\buildrel{pr^{1}}\over\To\P^{1}$. Then $\infty=(1:0)\notin R$ since $P_{\infty}$ is separable.

Let $\alpha':Y\to\P^{1}\times\P^{1}$ be the conic bundle given by $y^{2}-az^{2}=s'$. As a closed subvariety, $Y$ lies inside the projective plane bundle $\P(\E)\to\P^{1}\times\P^{1}$, where $\E=\O\oplus\O\oplus\O(1,2)$ is a $3$-dimensional vector sheaf on $\P^{1}\times\P^{1}$.
Denote by $\beta':Y\to\P^{1}$  the composition of $\alpha'$ and the first projection $pr^{1}$. Then the fiber  $\beta'^{-1}(\infty)$  is exactly the Ch\^{a}telet surface $Y_{\infty}$ at the beginning.

Choose a smooth projective curve $C$ over $F$ such that  $C(F)$ is finite and non-empty.   Take a dominant morphism $\gamma:C\to\P^{1}$ which is \'etale above $R$ and which maps all rational points to $\infty\in\P^{1}$. The \'etaleness is to ensure that $Z=Z'\times_{\P^{1}}C$ is a smooth curve. 
 Finally, take $\beta:X\to C$ to be the pull-back of $\beta':Y\to\P^{1}$ by $\gamma$.
Poonen showed that $X$ is a smooth proper geometrically integral variety over $F$ verifying the following property.
\begin{prop}[{\cite[Theorem 7.2]{Poonen}}]\label{poonenprop} 
The \BM set  $X(\A_{F})^{\Br}$ contains $Y_{\infty}(\A_{F})\times C(F)$.
\end{prop}
In the situation of \cite{Poonen}, the constant $a$ and the polynomial $P_{\infty}(x)$ are carefully chosen such that $Y_{\infty}(\A_{F})\neq\emptyset$ but $Y_{\infty}(F)=\emptyset$. It follows immediately that $X(F)=\emptyset$ but both $X(\A_{F})^{\Br}$ and $X(\A_{F})^{\rm{et},\Br}$ are non-empty.

\subsection{\BMo over $L$}\ \\
For our purpose, we specify $F$ and $Y_{\infty}$ now.
\begin{prop}\label{propL}
If $F=L$ and $Y_\infty$ is  the Ch\^{a}telet surface $V_{L}$ constructed in \S \ref{exWAsectiontitle}, then $X_{L}$ does not verify weak approximation with \BMo off $\infty_{L}$, i.e. $X(L)$ is not dense in $pr^{\infty}(X(\A_{L})^{\Br})\subset X(\A_{L}^{\infty})$.
\end{prop}

\begin{proof}
Suppose, for the sake of contradiction, that $X_{L}$ verifies weak approximation with \BMo off $\infty_{L}$. Take arbitrary family $(y_{w})\in V_{L}(\A_{L}^{\infty})$. As $V_{L}$ admits rational points at all archimedean places, we complete the family into $(y_{w})_{w\in\Omega_{L}}\in V_{L}(\A_{L})$. Take a point $P\in C(L)$, then $$(x_{w})_{w\in\Omega_{L}}=((y_{w})_{w\in\Omega_{L}},P)\in Y_{\infty}(\A_{L})\times C(L)\subset X(\A_{L})^{\Br}$$ by Proposition \ref{poonenprop}.
Let $S\subset\Omega_{L}\setminus\infty_{L}$ be any non-empty finite set of non-archimedean places. Then $(x_{w})_{w\in S}$ can be approximated by a rational points $x\in X(L)$. Therefore $(\beta(x_{w}))_{w\in S}=(P)_{w\in S}$ can be approximated by $\beta(x)\in C(L)$. As $C(L)$ is finite thus discrete in $\prod_{w\in S}C(L_{w})$, we find that $\beta(x)=P$. Whence $x\in\beta^{-1}(P)\simeq Y_{\infty}=V_{L}$ can approximate $(y_{w})_{w\in S}$, which contradicts to Propsition \ref{prop2}.
\end{proof}

\subsection{\BMo over $K$}\label{BMoverK}\ \\
For $F=L$, we explain how to make Poonen's construction be defined over the subfield $K$. And we will study its arithmetic over $K$.

As constructed in \S \ref{exWAsectiontitle}, $Y_{\infty}=V_{L}$ is actually obtained from the base change of a $K$-variety $V$ to $L$. If we also take $P_{0}(x)$ with coefficients in $K$, then the fibration $\beta':Y\to \P^{1}$ is defined over $K$. Now we need to find a smooth projective $K$-curve $C$  such that not only $C(K)$  but also  $C(L)$ is finite (and non-empty). According to Faltings' theorem, it is always the case if $C$ is of genus $\geq2$. But it seems unclear whether there exist such elliptic curves.

By Riemann\textendash Roch theorem, once $C(L)$ is a finite set of closed points of $C$, we can choose a rational function $\phi\in K(C)^{*}\setminus K^{*}$ such that $C(L)$ is contained in the set of poles of $\phi$. This give us a morphism $\gamma:C\to\P^{1}$ defined over $K$ mapping $C(L)$ to $\infty\in\P^{1}$.
Composing with a linear automorphism of $\P^{1}$ if necessary, we may require that $\gamma$ is \'etale above the finite set $R\subset \P^{1}$. 
Finally the pul-back $\beta:X\to C$ is also defined over $K$.

The following  conjecture  was stated by M. Stoll \cite[Conjecture 9.1]{Stoll07}. Similar questions were raised even earlier, independently by V. Scharaschkin \cite{Scharaschkin} and A. Skorobogatov \cite[\S 6.2]{Skbook}. The conjecture holds for $C=\P^{1}$, and according to Cassels\textendash Tate exact sequence it holds for elliptic curves $C=E$ provided that the Tate\textendash Shafarevich group $\sha(E,K)$ is finite.
\begin{conj}[Stoll]\label{Stollconj}
Let $C$ be a smooth projective curve defined over a number field $K$. Then $C$ verifies weak approximation with \BMo off $\infty_{K}$, i.e. $C(K)$ is dense in $pr^{\infty}(C(\A_{K})^{\Br})\subset C(\A_{K}^{\infty})$.
\end{conj}

\begin{prop}\label{propK}
The \BM set $X(\A_{K})^{\Br}$ is non-empty.
Assuming Conjecture \ref{Stollconj}, the variety $X_{K}$ verifies weak approximation with \BMo off $\infty_{K}$, i.e. $X(K)$ is dense in $pr^{\infty}(X(\A_{K})^{\Br})\subset X(\A_{K}^{\infty})$.
\end{prop}

\begin{proof}
As the fibration $X$ is defined over $K$, we apply Proposition \ref{poonenprop} with $F=K$ to find that $X(\A_{K})^{\Br}$ contains $Y_{\infty}(\A_{K})\times C(K)$ which is non-empty by Proposition \ref{prop1}.

For any $(x_{v})_{v\in\Omega_{K}}\in X(\A_{K})^{\Br}$,  we have $(\beta(x_{v}))_{v\in\Omega_{K}}\in C(\A_{K})^{\Br}$ by  functoriality. As $C(K)$ is finite and dense in $pr^{\infty}(C(\A_{K})^{\Br})\subset C(\A_{K}^{\infty})$.
There exists a rational point $P\in C(K)$ such that for any place $v\in\Omega_{K}\setminus\infty_{K}$ we have $P=\beta(x_{v})$. In other words,  $x_{v}$ lies on the fiber $\beta^{-1}(P)\simeq V_{K}$ for $v\notin\infty_{K}$. The fiber $\beta^{-1}(P)$ admits rational points at every archimedean place by the construction of $V_{K}$. According to Proposition \ref{prop1}, there exists a $K$-rational point of $\beta^{-1}(P)\subset X$ approximating $(x_{v})_{v\notin\infty_{K}}$ as desired.
\end{proof}

\subsection{The main result}\ \\
To summarise, we obtain the following  result.
\begin{thm}\label{exBMWA}
Let $K$ be a number field. There exists a 3-fold $X$ over $K$ and a quadratic extension $L$ of $K$ such that 
\newline \indent - $X(\A_{K})^{\Br}\neq\emptyset$, and  if Conjecture \ref{Stollconj} is assumed to hold over $K$ then $X_{K}$ verifies weak approximation with \BMo off $\infty_{K}$;
\newline \indent - $X_{L}$ does not verify weak approximation with \BM obstruction off $\infty_{L}$. 
\end{thm}

\begin{rem}
The statement of Theorem \ref{exBMWA} remains valid if we replace the \BM obstruction by the \'etale-\BM obstruction, or equivalently the (iterated) descent obstruction according to C. Demarche \cite{Demarche09}, A. Skorobogatov \cite{Sk09etBMdesc}, and Y. Cao \cite{Cao2017}.
Indeed, in the proof we only need to apply \cite[Theorem 8.2]{Poonen} instead of Proposition \ref{poonenprop}.
\end{rem}

\subsection{An unconditional explicit example over $\Q$}\label{Qexample}\ \\
For $K=\Q$, we give an unconditional example following the constructions in \S \ref{exWAsectiontitle} and \S \ref{exBMWAsectiontitle}. 
The example is given by gluing affine pieces defined by explicit equations.

\subsubsection{Ch\^{a}telet surfaces}
The Ch\^{a}telet surface  $V_{/\Q}$ is given by 
$$y^{2}-17z^{2}=P_{\infty}(x)=137(x^{4}+10x^{2}-155)$$ with $a=17$, $b=137$, $c=5$, $e=-31$, $d=ce=-155$.  Then $D=c^{2}-d=180$ and $L=K(\sqrt{D})=\Q(\sqrt{5})$. The variety $V_{\Q}$ and the extension $L|\Q$ satisfy Theorem \ref{exWA}. The fiber $Y_{\infty}$ that appears later will be birational to $V$. And  $Y_{0}$ will be birational to the Ch\^{a}telet surface defined by
$$y^{2}-17z^{2}=P_{0}(x)=-137(x^{4}-155).$$

\subsubsection{Conic bundle}
Let $s'\in\Gamma(\P^{1}\times\P^{1},\O(2,4))$ be the section 
$$\begin{array}{r@{\quad=\quad}l}
s'& u^{2}\tilde{P}_{\infty}(w,x) + v^{2}\tilde{P}_{0}(w,x) \\
& 137\left[u^{2}(x^{4}+10x^{2}w^{2}-155w^{4}) - v^{2}(x^{4}-155w^{4}) \right],\end{array}$$
where $\tilde{P}_{\infty}(w,x)$ and $\tilde{P}_{0}(w,x)$  are homogenizations of $P_{\infty}(x)$ and $P_{0}(x)$ and $(u:v,w:x)\in\P^{1}\times\P^{1}$ are homogeneous coordinates. Let $\mathcal{E}=\O\oplus\O\oplus\O(1,2)$ be a vector sheaf on $\P^{1}\times\P^{1}$. The closed subvariety $Y$,  defined by the ``equation'' $y^{2}-17z^{2}=s'$,
of the total space of the  projective plane bundle $\P(\mathcal{E})\to\P^{1}\times\P^{1}$ is a conic bundle over $\P^{1}\times\P^{1}$. We will explain the word ``equation'' later. The fiber over $\infty$ (respectively $0$) of the composition $\beta'=pr^{1}\circ\alpha':\P(\E)\to\P^{1}\times\P^{1}\to\P^{1}$ is the Ch\^{a}telet surface $Y_{\infty}$ (respectively $Y_{0}$) above.

\subsubsection{Degenerate locus and ramification}
The degenerate locus $Z'\subset\P^{1}\times\P^{1}$ of the conic bundle is defined by
$$s'=137\left[u^{2}(x^{4}+10x^{2}w^{2}-155w^{4}) - v^{2}(x^{4}-155w^{4}) \right]=0.$$
According to Jacobian criterion, it is smooth since $P_{0}(x)$ and $P_{\infty}(x)$ are separable and  coprime to each other.  Then $Y$ is also smooth.
$Z'$ projects onto $\P^{1}$ via the first projection. The branch locus $R\subset\P^{1}$ is a closed subscheme of dimension $0$. It consists of $6$ geometric points $\pm1$ and $\displaystyle\pm\frac{\sqrt{31\pm\sqrt{-155}}}{6}$. In particular, $\infty\notin R$.

\subsubsection{Base curve} Let $C=E$ be the elliptic curve over $\Q$ defined by 
$$y'^{2}z'=x'^{3}-4x'z'^{2}$$
with homogeneous coordinates $(x':y':z')\in\P^{2}$. The CM elliptic curve $E$ and its quadratic twist $E^{(5)}$ by $L|\Q$ both have  $L$-functions that are non-vanishing at the complex $1$. According to K. Rubin \cite{Rubin87}, the Tate\textendash Shafarevich group $\sha(E,\Q)$ is finite. Stoll's Conjecture \ref{Stollconj} thus holds for $E_{/\Q}$.
According to J. Coates and A. Wiles \cite{CoatesWiles77}, the Mordell\textendash Weil groups  $E(\Q)$ and $E^{(5)}(\Q)$ are both finite, whence so is $E(L)$. Indeed, $E(L)=E(\Q)=\Z/2\Z\oplus \Z/2\Z=\{(0:1:0),(0:0:1),(\pm2:0:1)\}$.

\subsubsection{Base change morphism} We are defining a non-constant morphism $\gamma:E\to\P^{1}$ \'etale over $R\subset\P^{1}$ such that rational points in $E(L)$ are all mapped to $\infty$. 

Define $\gamma$ by $(x':y':z')\mapsto(y'^{2}+z'^{2}:x'y')$. It extends to $\P^{2}\setminus D\to\P^{1}$ where $D$ is a $0$-dimensional closed subscheme  consisting of $3$ geometric points $(1:0:0)$ and $(0:\pm\sqrt{-1}:1)$. It maps all points in $E(L)$ to $\infty$.

According to Jacobian criterion, one can check by hand that $\gamma$ is \'etale over $\pm1\in R$. It is much more complicated, but one can check with the help of a computer that  $\gamma$ is also \'etale over $\pm\frac{\sqrt{31\pm\sqrt{-155}}}{6}\in R$. Practically, it suffices to check only one of the four since they are conjugate by the Galois action. With some tricks in the explicit calculation, finally, the \'etaleness  reduces to decide whether two polynomials of degree $6$ and $12$ with integral coefficients (11 decimal digits) are coprime to each other.

\subsubsection{Ch\^{a}telet surface bundle}
As $\gamma:E\to\P^{1}$ is \'etale over the branch locus $R$ of $Z'\to\P^{1}$, the fiber product $Z=E\times_{\P^{1}}Z'$ is then smooth. Our $\Q$-variety $X$ is defined to be the pull-back of $Y$ by $\gamma$. Viewed as a conic bundle over $E\times\P^{1}$, the smooth variety $Z$ is its degenerate locus. Hence $X$ is also smooth by Jacobian criterion since the coefficients of $y^{2}$ and $z^{2}$ are non-zero constants.

The variety $X$ is also the restriction of the pull-back $\mathcal{X}$ of $Y$ by 
$$(\gamma,1):(\P^{2}\setminus D)\times\P^{1}\to\P^{1}\times\P^{1}$$
to $E\times\P^{1}\subset(\P^{2}\setminus D)\times\P^{1}$. It is defined explicitly by the following ``equations''.
$$\left\{ \begin{array}{r@{\quad=\quad}l}
y^{2}-17z^{2} &137t^{2}\left[(y'^{2}+z'^{2})^{2}(x^{4}+10x^{2}w^{2}-155w^{4}) - x'^{2}y'^{2}(x^{4}-155w^{4}) \right]\\
y'^{2}z' & x'^{3}-4x'z'^{2}\end{array} ,\right. $$
The second  is the equation of $E\subset\P^{2}$ with homogeneous coordinates $(x':y':z')$. 

The first ``equation'' needs more explanation. 
It is the ``equation'' of $\mathcal{X}$ viewed as a subvariety of $\P(\E')$. Here $\P(\E')\to(\P^{2}\setminus D)\times\P^{1}$ is the projective bundle  associated to $\E'=(\gamma,1)^{*}\E$.  The vector sheaf $\E'$ equals to $\O\oplus\O\oplus\O(2,2)$ if we identify $\Pic(\P^{2}\setminus D)$ with $\Pic(\P^{2})$. 
If one denotes by $s$ the section 
$$137\left[(y'^{2}+z'^{2})^{2}(x^{4}+10x^{2}w^{2}-155w^{4}) - x'^{2}y'^{2}(x^{4}-155w^{4})\right]\in\Gamma(\P^{2}\setminus D,\O(2,2)^{\otimes2})$$
where $(x':y':z',w:x)$ are coordinates of $(\P^{2}\setminus D)\times\P^{1}$, then $\mathcal{X}$ is the zero locus in $\P(\E')$ of   
$$1\oplus-17\oplus -s\in\Gamma(\P^{2}\setminus D,\O\oplus\O\oplus\O(2,2)^{\otimes2})\subset\Gamma(\P^{2}\setminus D,\textup{Sym}^{2}(\E')).$$
So $\mathcal{X}$ is given by the first ``equation'' with homogeneous ``coordinates'' $(y:z:t)$ of the fibers of the projective plane bundle.

We make precise what ``equation'' means.
Restricted to one of the affine subsets of $(\P^{2}\setminus D)\times\P^{1}$ given by non-vanishing of one of the coordinates of each factor, the vector sheaf $\E'$ is trivialised. Only on such a trivial fibration, the coordinates $(y:z:t)$ of the fiber $\P^{2}$ make sense.
For example, restricted to the affine $(\mathbb{A}^{2}\setminus D)\times\mathbb{A}^{1}$ given by $w\neq0$ and $z'\neq0$. We get  equations 
$$\left\{ \begin{array}{r@{\quad=\quad}l}
y^{2}-17z^{2} &137t^{2}\left[(1+y'^{2})^{2}(x^{4}+10x^{2}-155) - x'^{2}y'^{2}(x^{4}-155) \right]\\
y'^{2} & x'^{3}-4x'\end{array} ,\right. $$
describing  a Zariski open dense subset of  $X$ lying inside $(\mathbb{A}^{2}\setminus D)\times\mathbb{A}^{1}\times\P^{2}$ by dehomogenization.
Then $X$ is given by gluing  $6$ similar explicit quasi-projective varieties via obvious isomorphisms on their intersections.
From such equations, we get the affine piece presented in the introduction via a further dehomogenization by taking $t=1$. 

By the way, $(x', y', x,\ y:z:t)=(0, 0, 1,\ 48:36:1)$ is an explicit $\Q$-rational point on $X$, on which the existence of rational points is also clear from the construction.

\subsection{Quadratic  number fields}\ \\
Though the conclusion over $K$ of Theorem \ref{exBMWA} is conditional in general, it is not difficult to prove unconditional results over a specific quadratic  field  with the help of the results on elliptic curves by  B. Gross and D. Zagier \cite{GrossZagier}, V. A. Kolyvagin \cite{Kolyvagin}, and many others.

The following lemma should be well known, we state it here for the convenience of the reader.

\begin{lem}\label{easylemma}
Let $E$ be an elliptic curve defined over a number field $k$. Let $k(\sqrt{a})$ be a quadratic extension of $k$. Denote by $E^{(a)}$ the quadratic twist of $E$. 
\begin{itemize}
\item[-] If both $\sha(E,k)$ and $\sha(E^{(a)},k)$ are finite, then $\sha(E,k(\sqrt{a}))$ is also finite.
\item[-] If both $E(k)$ and $E^{(a)}(k)$ are finite, then $E(k(\sqrt{a}))$ is also finite.
\end{itemize}
\end{lem}
\begin{proof}
It follows from the fact that $\textup{Res}_{k(\sqrt{a})|k}E_{k(\sqrt{a})}$ and $E\times E^{(a)}$ are isogenous \cite[page 185, Example 1]{Milne72} and the fact that the finiteness of the Tate\textendash Shafarevich group depends only on the isogenous class \cite[Lemma I.7.1(b)]{MilneADT}. 
\end{proof}

To the knowledge of the author, the only proved case of the following Bunyakovsky's Conjecture \ref{BunyakovskyConj} is 
Dirichlet's theorem on arithmetic progressions. As a generalisation of the conjecture to several polynomials, Schinzel's hypothesis (H) has turned out to be a crucial point to the question on  \BM obstruction, particularly for the fibration method,  see the work of Colliot-Th\'el\`ene and Sansuc \cite{CTSansuc82} and many other subsequent papers.
Bunyakovsky's conjecture  also plays a role in our constructions but in a different way.

\begin{conj}[Bunyakovsky]\label{BunyakovskyConj}
Let $f(x)\in\Z[x]$ be an irreducible polynomial of positive degree and of positive leading coefficient. 
Suppose that there does not exist a prime number dividing $f(n)$ for all $n\in\mathbb{N}$.
Then $f(n)$ has infinitely many prime values.
\end{conj}

Start with a quadratic number field $K$. We want to prove the existence of a $K$-variety $X$ and a quadratic extension $L$ of $K$ verifying  \emph{unconditionally} the conclusions of Theorem \ref{exBMWA}.

According to the proof of Theorem \ref{exBMWA}, it suffices to show the existence of parameters $a,b,c,e\in\O_{K}$ satisfying the conditions in \S \ref{exWAsectiontitle} and  an elliptic curve defined over $K$ such that Stoll's Conjecture \ref{Stollconj} holds for $E$ and such that $E(K(\sqrt{c^{2}-ce}))$ is finite.

To fix notation, let $K=\Q(\sqrt{\delta_{0}})$ with $\delta_{0}$ a square-free integer. The ring of integers $\O_{K}=\Z[x]/(\varphi)$ where $\varphi$ is a degree $2$ monic polynomial.

First, we need to find an elliptic curve $E$ over $\Q$ such that both $E$ and its quadratic twist $E^{(\delta_{0})}$ have $L$-functions that are non-vanishing at $s=1$. Thanks to Kolyvagin  \cite{Kolyvagin}, Gross and Zagier \cite{GrossZagier}, the   Tate\textendash Shafarevich groups and the Mordell\textendash Weil groups of these two curves are finite. Then $\sha(E,K)$ and $E(K)$ are finite by Lemma \ref{easylemma}. Stoll's Conjecture \ref{Stollconj} holds for $E_{K}$.

Second, we choose a square-free rational integer $\delta$ such that 
\begin{itemize}
\item[-] at least one prime factor $p$ of $\delta$  is odd and inert in $K$;
\item[-] both $E^{(\delta)}$ and $E^{(\delta\delta_{0})}$ are of  analytic rank $0$.
\end{itemize}
The first condition is equivalent to the irreducibility of $\varphi$ modulo $p$ which can be easily tested by quadratic reciprocity. If one believes the rank part of the BSD conjecture, the second condition is requiring  that $E^{(\delta)}$ and $E^{(\delta\delta_{0})}$ are of Mordell\textendash Weil rank $0$. According to B. Mazur and K. Rubin \cite[Corollary 1.9]{MazurRubin}, there are many quadratic twists of $E$ that have Mordell\textendash Weil rank $0$, but now we need to twist simultaneously  $E$ and $E^{(\delta_{0})}$ by a certain $\delta$. In practice, we often find such an integer $\delta$ satisfying these two conditions even in the set of prime numbers. Then $E(L)$ is finite by Lemma \ref{easylemma} where $L=K(\sqrt{\delta})$.

Finally, it remains to show the existence of $a,b,c,e$ satisfying the conditions in \S \ref{exWAsectiontitle} and such that $L=K(\sqrt{\delta})=K(\sqrt{c^{2}-ce})$.
Define $f\in\Z[x]$ to be $\displaystyle f(x)=\frac{-\delta}{c}x^{2}+c$ by taking $c=p$ or $-p$ such that $-\delta/c$ is positive. This polynomial is of content $1$ and is irreducible by Eisenstein's criterion.
If a natural prime $l$ divides $f(n)$ for all $n\in\mathbb{N}$, then $l\leq\deg(f)=2$ by \cite[Lemma 2]{CTSansuc82}. 
But $l|f(l)$ implies that $l|c=\pm p$, which is  impossible since $p$ is odd.
Therefore $f(x)$ verifies the assumption of  Bunyakovsky's conjecture, which predicts the existence of many natural numbers $n$ such that  $e=f(n)$ is a  prime  larger than $p$. 
If, moreover, one of such $e$'s is inert in $K$, we claim that the existence of the desired variety $X$ is ensured. Indeed,  $L=K(\sqrt{\delta})=K(\sqrt{c^{2}-ce})$ since $\delta n^{2}=c(c-e)$. And $c,e$ generate different prime ideals in $\O_{K}$. By Proposition \ref{dirichlet} and Lemma \ref{reciprocity}, we can complete
them to  parameters $a,b,c,e\in\O_{K}$ satisfying all the conditions in \S \ref{exWAsectiontitle}.

With this strategy,  we prove the existence in Theorem \ref{exBMWA} with unconditional conclusion over quadratic number fields $K$ in the following table (not an exhaustive list). In practice, even if in a very unlucky situation:  $\delta$ does not have a prime factor that is inert in $K$, one still has chance to find appropriate parameters, see the last line of the table.
\medskip

\begin{tabular}{ | l | c | c | c | c | l |}
    \hline
    \ \ \ \ $K$ & $\delta$  & $c$ & $e$ & $n$ & \ \ \ \ \ $E$ \\ 
    \hline\hline
    $\Q(\sqrt{3})$ & $11$ & $-11$ & $5$ & $4$ & $y^{2}=x^{3}-x$\\ 
    \hline
    $\Q(\sqrt{-3})$ & $-11$ & $11$ & $47$ & $6$ & $y^{2}=x^{3}-x$\\
    \hline
    $\Q(\sqrt{-19})$ & $-3$ & $3$ & $67$ & $8$ & $y^{2}=x^{3}-x$\\
    \hline
    $\Q(\sqrt{-5})$ & $13$ & $-13$ & $131$ & $12$ & $y^{2}=x^{3}-4x$\\
    \hline
    $\Q(\sqrt{-1})$ & $5$ & $2+\sqrt{-1}$ & $-6+5\sqrt{-1}$ & $2$ & $y^{2}=x^{3}-4x$\\
    \hline
\end{tabular}

\section{Remarks on related questions}
\subsection{Remarks on the other direction}
In the other direction to answer negatively to Question \ref{mainquestion}, we also seek examples that fail a certain weak approximation property over the ground field but verify the property over a certain finite extension. It turns out that existing examples feed our needs.

Over an arbitrary number field $K$, due to Poonen \cite[Proposition 5.1]{Poonen09} there exists a Ch\^atelet surface $V$ which violates the Hasse principle (and explained by the \BM obstruction), a fortiori it fails weak approximation. Suppose that $V$ is defined by $$y^{2}-az^{2}=P(x).$$
Then $a\in K^{*}$ must not be a square and the polynomial $P(x)\in K[x]$ must be reducible over $K$ \cite[Theorem 8.11]{chateletsurfaces}.
Let $L$ be the quadratic extension $K(\sqrt{a})$ of $K$. Then $V$ is a rational variety over $L$, it verifies weak approximation over $L$.

Using $V$ as $Y_{\infty}$ in the construction of \S \ref{poonenconstruction} with $F=K$, Poonen show that $X_{K}$ violates Hasse principle with (\'etale-)\BM obstruction over $K$, a fortiori it fails weak approximation with (\'etale-)\BM obstruction off $\infty_{K}$. Let $L$ be as above. If the base curve $C$ admits only finitely many $L$-rational points which are all mapped to $\infty\in\P^{1}$ by $\gamma$, then one can argue as in Proposition \ref{propK} to show that $X_{L}$ verifies weak approximation with (\'etale-)\BMo off $\infty_{L}$ assuming Conjecture \ref{Stollconj} over $L$. As explained in \S \ref{Qexample}, for $K=\Q$ the last statement is unconditional for a certain elliptic curve.

Note that in these examples neither $V$ nor $X$  possesses $K$-rational points. One may ask for examples possessing $K$-rational points. Indeed, it is even easier, it suffices to define the Ch\^{a}telet surface $V$ by $$y^{2}-az^{2}=-(x^{2}+b)(x^{2}-b)$$
with $a,b\in\O_{K}$ generating different prime ideals such that  $a>0$ with respect to all real places.
Then  $(x,y,z)=(0,b,0)$ is a rational point  on $V$ and hence $X(K)$ is non-empty. We know that $X_{K}$ fails the \BMo off $\infty$ as long as  $V_{K}$ does not verify weak approximation off $\infty_{K}$. To show the latter, the key argument is similar to that in Proposition \ref{prop2}:
if $A\in\Br(K(V))$ is the class of the quaternion algebra $(a,x^{2}+b)$, then the evaluation of $A$ on $V(K_{v_{b}})$ takes both $0$ and $1/2\in\Q/\Z$ as values (for the value $1/2$, one considers $x=b$). 
It is also clear that over the extension $L=K(\sqrt{a})$, $V_{L}$ is a rational variety and $X_{L}$ satisfies weak approximation with \BM obstruction if $C(L)$ is finite and mapped to $\infty$.

\subsection{Over global function fields}
Poonen's construction also works over odd characteristic global function fields  and B. Viray extended the result to characteristic $2$ \cite{Viraychar2}.  Moreover, if $C$ is a smooth projective curve of genus at least $2$ whose Jacobian $J$ satisfies some mild conditions, Conjecture \ref{Stollconj} for $C$ has been proved by B. Poonen and J. F. Voloch \cite{PoonenVoloch} even without finiteness assumption of the Tate\textendash Shafarevich group of $J$. As explained to the author by Voloch, his paper \cite{Voloch95} showed that those conditions on $J$ are satisfied for  many curves.  Poonen can explicitly construct  such genus $2$ curves in odd characteristic case using the method in \cite[\S 3.2]{HLP}. We may expect our main result to hold unconditionally over a global function field (at least in odd characteristic case). However, the proof of Proposition \ref{prop1}  uses the fact that the \BMo is the only obstruction to weak approximation for Ch\^atelet surfaces. This result was stated in \cite[Theorem 8.11]{chateletsurfaces} only for number fields. It is not clear if the analog holds over odd characteristic global function fields, or even more difficult in characteristic $2$.

\subsection{A related question}
Though the answer to Question \ref{mainquestion} is negative, we may still ask the following question.
\begin{ques}
Under which finite extensions, weak approximation properties are invariant?
\end{ques}
The question is easy for the families of varieties  $X_{K}$ in \S \ref{exBMWAsectiontitle} and $Y_{\infty,K}$ in  \S \ref{BMoverK} (or $V_{K}$ in \S \ref{exWAsectiontitle}). Over any finite extension $K'$ that is linearly disjoint from $K[x]/(P_{\infty}(x))$ over $K$, the polynomial $P_{\infty}(x)$ is still irreducible.
Then the varieties $V_{K'}=Y_{\infty,K'}$ and $X_{K'}$ preserve  the same property as on the ground field (Propositions \ref{prop1} and \ref{propK}). 
But it seems hopeless to answer to the question in general.

\bigskip

\begin{footnotesize}
\noindent\textbf{Acknowledgements.}  The author would like to thank B. Poonen, B. Viray, and J. F. Voloch for having explained  their work to the author.


\end{footnotesize}


\bigskip

\bibliographystyle{alpha}
\bibliography{mybib1}

\end{document}